\documentclass[a4paper,12pt]{amsart}
\usepackage{palatino,hhline, bbm}
\usepackage{arydshln} 
\usepackage{amsthm, amsmath}
\usepackage{amssymb} 
\usepackage{amscd}
\usepackage{mathrsfs}
\renewcommand{\mathcal}{\mathscr}
\usepackage[left=1.2cm,right=1.2cm,  top=1.5cm,bottom=1.7cm,bindingoffset=0cm]{geometry}
\usepackage[breaklinks,bookmarksopen,bookmarksnumbered]{hyperref}
\usepackage[all]{xy}

\usepackage{mathtools}

\usepackage{multirow}

\usepackage{ amsfonts, mathtext, cite, enumerate, float}

\usepackage{graphicx}

\newtheorem{theorem}{Theorem} [section]
\newtheorem{lemma}[theorem]{Lemma} 
 
\newtheorem{defin}[theorem]{Definition}

\newtheorem{cor}[theorem]{Corollary}

\begin{document}

\title{Quotients of Jacobians}
\author[Raisa Serova]{Raisa Serova}

\email{serova.r@yandex.ru}

\begin{abstract}
	We prove that the quotient of Jacobian of a curve whose genus is greater than or equal to 5 under the action of a finite group acting on the curve is never uniruled, and classify all curves of genus 3 and 4 whose quotients of Jacobian is uniruled.
  \end{abstract}
\maketitle

\section{Introduction}\label{intro}

\normalsize{\, \,  Let $C$ be a compact Riemann surface of genus $g$, and let $G$ be a finite group of automorphisms of $C$. The group $G$ acts on the Jacobian $J$ of $C$. There is a question for which $g$ quotient $J/G$ has Kodaira dimension 0 in A. Beauville's article~\protect\cite{AB}. According to \cite[Theorem 2]{KL} this is equivalent to $J/G$ not to be uniruled. 

\begin{theorem}[\protect{\cite[Proposition 1]{AB}}]\label{b1} Assume $g \geq 21$. The quotient variety $J/G$ is not uniruled.\end{theorem}

\begin{theorem}[\protect{\cite[Proposition 3]{AB}}]\label{b2} Assume $g = 5$. The quotient variety $J/G$ is not uniruled.\end{theorem}
}

{
For $g = 2$ there is a criterion in \cite{AB} that allows to determine whether the  quotient variety $J/G$ is uniruled or not. Assume  $\mathbf{i}^{2} = -1$ and $\omega = e^{\frac{2\pi \mathbf{i}}{3}}$.

\begin{theorem}[\protect{\cite[Proposition 2]{AB}}]\label{b3} If $g = 2$, then  $J/G$ is uniruled, except for the cases when G is a cyclic group generated by $\sigma$ and $\sigma$ is hyperelliptic involution or an automorphism of order 3 with eigenvalues $(\omega, \omega^{2})$ on $H^{0, 1}(C) \cong H^{0}(C, K_{C})$ or an automorphism of order 6 with eigenvalues $(-\omega, -\omega^{2})$ on $H^{0, 1}(C)$. \end{theorem}

}

{
The main results of this work are the following theorems. Let $C$ be a compact Riemann surface of genus $g$, and let $G$ be a finite group of automorphisms of  $C$. The group $G$ acts on the Jacobian $J$ of compact Riemann surface $C$. Let us denote the order of automorphism $\sigma$ of $G$ by $N$.

\begin{theorem}\label{m1}Assume $g \geq 5$. The quotient variety $J/G$ is not uniruled.\end{theorem}

For $g = 4$ we will prove the criterion that allows to determine whether the quotient variety $J/G$ is uniruled or not.

\begin{theorem}\label{m2}Assume $g = 4$, then $J/G$  is uniruled if and only if G contains an element of one of the following orders and  compact Riemann surface is isomorphic to the following one:

\begin{center}
    \begin{minipage}{0.95\textwidth}
      \begin{enumerate}
	\item{$N = 15$, compact Riemann surface is isomorphic to $y^{3} = x(x^{5} - 1)$ and  $\sigma(x, y) = (\zeta^{3} x, \zeta y)$,  where $\zeta = e^{\frac{2\pi \mathbf{i}}{15}}$;}
	\item{$N = 18$, compact Riemann surface is isomorphic to $y^{2} = x(x^{9} - 1)$ and  $\sigma(x, y) = (\zeta^{2} x, \zeta y)$,  where $\zeta = e^{\frac{2\pi \mathbf{i}}{18}}$.}
      \end{enumerate}
    \end{minipage}
  \end{center}

\end{theorem}
}

{For $g = 3$ we will prove the criterion that allows to determine whether the quotient variety $J/G$ is uniruled or not.

\begin{theorem}\label{m3}Assume $g = 3$, then $J/G$  is uniruled if and only if G contains an element of one of the following orders with following eigenvalues on $H^{0, 1}(C)$:

\begin{center}
    \begin{minipage}{0.88\textwidth}
      \begin{enumerate}
	\item{N = 2, eigenvalues $\sigma$: $1, 1, -1$;}
	\item{N = 7, eigenvalues $\sigma$: $\zeta, \zeta^{2}, \zeta^{3}$, where $\zeta = e^{\frac{2\pi \mathbf{i}}{7}}$;}
	\item{$N = 8$, eigenvalues $\sigma$: $\zeta, \zeta^{2}, \zeta^{3}$,  where $\zeta = e^{\frac{2\pi \mathbf{i}}{8}}$;}
	\item{$N = 9$,  eigenvalues $\sigma$: $\zeta, \zeta^{2}, \zeta^{4}$,  where $\zeta = e^{\frac{2\pi \mathbf{i}}{9}}$;}
	\item{$N = 12$, eigenvalues $\sigma$: $\zeta, \zeta^{3}, \zeta^{5}$, or  $\zeta, \zeta^{2}, \zeta^{5}$, where $\zeta = e^{\frac{2\pi \mathbf{i}}{12}}$;}
	\item{$N = 14$, eigenvalues $\sigma$: $\zeta, \zeta^{3}, \zeta^{5}$,  where $\zeta = e^{\frac{2\pi \mathbf{i}}{14}}$.}
      \end{enumerate}
    \end{minipage}
  \end{center}

All these cases are realized.

\end{theorem}
} 

\begin{cor}\label{c1}Assume $g = 4$, then quantity of compact Riemann surfaces such that $J/G$ is uniruled is finite, if $g = 2$ or $g = 3$, then quantity of such curves is infinite.
\end{cor}

The most symmetric curves of genus  $3$ and $4$ are Klein's quartic and Bring's curve. Bring's curve is the curve cut out by the homogeneous equations $$u + v + w + t + s = u^{2} + v^{2} + w^{2} + t^{2}+ s^{2}  = u^{3} + v^{3} + w^{3} + t^{3} + s^{2} = 0$$ in $\mathbb{P}^{4}$. This is the only curve of genus $g = 4$ with  $S_{5}$ group action (see \cite{W}).
Klein's quartic is defined by the following quartic equation $$u^{3}v + v^{3}w + w^{3}u = 0$$ in $\mathbb{P}^{2}$. This is the only curve of genus $g = 3$ with automorphism group $\mathrm{PSL}_{2}(\mathbb{F}_{7})$ (see \cite{Klein}).

\begin{cor}\label{B_K}
The following assertions hold
\begin{enumerate} 
\item{The quotient of Jacobian of Bring's curve by its automorphism group is not uniruled.}
\item{The quotient of Jacobian of Klein's curve by its automorphism group is not uniruled.}
\end{enumerate}
\end{cor}

{Quotients of Jacobians by their automorphism groups have been studied already. For instance, in \cite{MM} it is shown that the quotient of Jacobian of Klein's curve by group of order~$336$ is isomorphic to weighted projective space $\mathbb{P}(1, 2, 4, 7)$. Note that the second assertion of the Corollary \ref{B_K} claims that quotient is not uniruled under the action of subgroup or order $168$.}

{There is a list of all automorphism groups of compact Riemann surface of genus $4$ in \cite{KK} (\cite[Proposition 2.1]{KK}). Nevertheless, detailed proof of this fact is omitted. We will not use the result of this work and we will prove Theorem \ref{m2} regardless of this list. Additionally, there is a list of automorphism groups of compact Riemann surface of genus $3$ in \cite{KK} (\cite[Proposition 1.1]{KK}). We will not use it in the proof of  \ref{m3}.}

{In Section \ref{background} we will provide some preliminary information needed in theorem proofs. In Section \ref{big} we will prove Theorem \ref{m1}; the proof follows proof of Theorem \ref{b1}, provided in \cite{AB}, but with more accurate estimates on eigenvalues of action of first cohomology group of compact Riemann surface. In Section \ref{four} we will prove Theorem \ref{m2}. In Section \ref{three} we will prove Theorem \ref{m3}. In Section \ref{Cor} we will prove corollaries \ref{c1} and \ref{B_K}.}

\section{Background}\label{background}

It is well known that every irreducible (non-campact) Riemann surface has the only one non-singular irreducible compactification. All our Riemann surfaces are compact and irreducible; when we say that Riemann surface $C$ is defined by equations in affine space, we mean that $C$ is a smooth irreducible compactification of non-compact Riemann surface, defined by these equations.

\begin{defin} \textnormal{Let $G$ be a finite group, and let $(\rho, V)$ be a $g$-dimensional representation of $G$. We say that representation $\rho : G \rightarrow GL(V)$ satisfies the \textit{local Reid condition} if $$a_{1}+\ldots+a_{g} \geq r,$$ where $r$ is order of $\sigma$, and $e^{\frac{2\pi \mathbf{i} a_{i}}{r}}$ are eigenvalues $\rho(\sigma)$ and $0 \leq a_{i} < r$,  for every $\sigma$ in $G$.}\end{defin}

\begin{defin} \textnormal{Let $G$ be a finite group acting on a smooth projective variety $X$. We say that the $G$-action satisfies the \textit{global Reid condition} if for every  $x$ from $X$, the representation of stabilizer of $x$ in $T_{x}(X)$ satisfies the local Reid condition.}
\end{defin}


\begin{theorem}[\protect{\cite[Theorem 2]{KL}}]\label{kl1} Let $X$ be  a smooth variety with trivial canonical bundle and $G$ a finite group acting on $X$. The following are equivalent:
 \begin{enumerate}
\item{The G-action satisfies the global Reid condition ;}
\item{$X/G$ is not uniruled.}
\end{enumerate}
\end{theorem}

\begin{lemma}
Let $\sigma$ be a non-trivial automorphism of finite order of a compact Riemann surface $C$ of genus $g \geq 1$, then eigenvalues $\sigma$ on $H^{0, 1}(C)$ are not equal to unity.
\end{lemma}

\begin{proof}
Action of $\sigma$ on Jacobian is non-trivial by the Torelli Theorem. The element of finite order, which acts on the Jacobian, fixes zero point, \cite[\S 2.2, Identity Theorem]{A}. We identify tangent space at zero with $H^{0, 1}(C)$. It means that eigenvalues cannot be equal to unity.
\end{proof}

Further we use this theorem without mentioning it.

\begin{theorem}[\protect{\cite[Corollary 9.6]{TB}}]\label{order}
The maximum order of an automorphism of a compact Riemann surface of genus $g \geq 2$ is $4g + 2$.
\end{theorem}

\begin{theorem}[\protect{\cite[V.2.11]{FK}}]\label{fk} Let $\sigma$ be automorphism of prime order of a compact Riemann surface $C$ of genus $g$. If $\sigma$ has fixed point, then it has at least two fixed points.
\end{theorem}

\begin{theorem}[\protect{\cite[Lemma 5]{FB}}]\label{fb} Let $\sigma$ be an automorphism of prime order of a compact Riemann surface $C$ of genus $g$. Then either $p \leq g$ or  $p = g + 1$ or $p = 2g + 1$. \end{theorem}

\begin{theorem}[\protect{Lefschetz Fixed Point formula, \cite[Corollary 12.3]{TB}}] Let $\sigma$ be an automorphism of order $r$ of a compact Riemann surface $C$ of genus $g \geq 2$, and $\zeta^{a_{1}},\ldots , \zeta^{a_{g}}$ are its eigenvalues. Then 
$$\zeta^{a_{1}} + \ldots + \zeta^{a_{g}} + \zeta^{-a_{1}} + \ldots + \zeta^{-a_{g}} = 2 - |\mathrm{Fix}_{C}(\sigma)|,$$ 
where $\mathrm{Fix}_{C}(\sigma)$ is a set of fixed points of automorphism $\sigma$.\end{theorem}

\begin{lemma}[\protect{\cite[Example 30.3]{TB}}]\label{example}If $\sigma$ is an automorphism of order $3$ of a compact Riemann surface $C$ of genus $g \geq 2$, and $1, \omega, \omega^{2}$  are eigenvalues of $\sigma$ on $H^{0, 1}(C)$. Let the number of eigenvalues be equal to $k_{0}, k_{1}, k_{2}$ respectively, then $k_{0}, k_{1}, k_{2}$ must satisfy the following inequalities
\begin{equation*}
 \begin{cases}
k_{0}, \, \, k_{1}, \, \, k_{2} \geq 0; \\
k_{0} + k_{1} - 2k_{3} \leq 1;\\
k_{0} - 2k_{1} + k_{2} \leq 1;\\
k_{0} + k_{1} + k_{3} \geq 2.
 \end{cases}
\end{equation*}
\end{lemma}

Let $I(m)$ be the set of positive integers smaller than $m$ that are coprime to $m$.

\begin{defin}[\protect{The Condition $\widetilde{E}$, \cite[Definition 30.1]{TB}}]\label{e} \textnormal{We say that the number $\alpha \in \mathbb{Q}(\zeta_{m})$ satisfies the Eichler Trace Formula for $m$, if $\alpha$ can be written as $$\alpha = 1 + \sum\limits_{u \in I(m)}f_{u}\frac{\zeta^{u}_{m}}{1 - \zeta^{u}_{m}}, $$ with non-negative integers $f_{u}$. We say that $G$-character $\chi$ satisfies the condition $\widetilde{E}$ if $\chi(1) \geq 2$ and for all~$\sigma \in G$, $\chi(\sigma)$ satisfies the Eichler Trace Formula for $|\sigma|$.}
\end{defin}

Let $G$ be a finite cyclic group and $\sigma \in G$, and  $CY(G, \sigma)$ the set of those cyclic subgroups of $G$ that contain $\langle \sigma \rangle$ properly. For each group in $CY(G, \sigma)$, choose a generator $k$ such that $k^d = \sigma$ for $d = |k|/|\sigma|$. Denote the set of these $k$ by $cy(G, \sigma)$. Note that this set is not uniquely defined.
\begin{defin}[\protect{The Condition $\widetilde{RH}$, \cite[Definition 30.8]{TB}}]\label{rh} \textnormal{Let $G$ be a finite group, and $\chi$ a class function of $G$ that satisfies $\widetilde{E}$. If non-negative integers $r_{\chi, u}(\sigma)$ are given such that 
$$\chi(\sigma) = 1 + \sum\limits_{u \in I(|\sigma|)}r_{\chi, u}(\sigma)\frac{\zeta^{u}_{m}}{1 - \zeta^{u}_{m}}$$ holds for all $\sigma \in G\backslash \{\mathrm{id}\}$ then the values $$r^{*}_{\chi, u}(\sigma) = r_{\chi, u}(\sigma) -  \sum\limits_{k \in cy(G, \sigma)}\sum\limits_{\substack{v \in I(|k|) \\ v \equiv u (mod |\sigma|)}}r^{*}_{\chi, v}(k)$$
and 
$$l_{\chi, u}(\sigma) = \frac{r^{*}_{\chi, u}(\sigma)}{[C_{G}(\sigma):\langle \sigma \rangle]}$$ are well-defined. In this case, we call $(l_{\chi, u}(\sigma))_{\sigma \in G\backslash \{\mathrm{id}\}, u \in I(|\sigma|)}$  an admissible system for $\chi$ if the underlying map $r_{\chi, u}$ on $G\backslash \{\mathrm{id}\}$ is constant on conjugacy classes of $G$ and $r_{\chi, u}(\sigma^{u})~=~r_{\chi, 1}(\sigma)$ holds for all $\sigma \in G\backslash \{\mathrm{id}\}$ and $ u \in I(|\sigma|)$.
We say that $\chi$ satisfies the condition $\widetilde{RH}$ if and only if $\chi$ satisfies the condition $\widetilde{E}$, and there is an admissible system $(l_{\chi, u}(\sigma))_{\sigma \in G\backslash \{\mathrm{id}\}, u \in I(|\sigma|)}$ such that $l_{\chi, u}(\sigma)$ is a non-negative integer for all $\sigma \in G\backslash \{\mathrm{id}\}$ and $u \in I(|\sigma|)$.}
\end{defin}

\begin{theorem} [\protect{\cite[Corollary 32.2]{TB}}]\label{cr_e} Let $G$ be a cyclic group, $\chi$ is a $G$-character. Then the following statements are equivalent:
\begin{enumerate}
\item{Character $\chi$ satisfies the condition $\widetilde{RH}$.}
\item{There is a compact Riemann surface $C$ of genus $g = \chi(1)$ with $G$ group action such that the representation of group in $H^{0, 1}(C)$ defined by $\chi$.}
\end{enumerate}
\end{theorem}

\begin{theorem}[\protect{\cite[Theorem 1]{K}}]\label{cr_e_p}
Let $G$ be a cyclic group of a prime order, $\chi$ is a $G$-character. Then the following statements are equivalent:
\begin{enumerate}
\item{Character $\chi$ satisfies the condition $\widetilde{E}$.}
\item{There is a compact Riemann surface $C$ of genus $g  = \chi(1)$ with $G$ group action such that the representation of group in $H^{0, 1}(C)$ is defined by $\chi$.}
\end{enumerate}
\end{theorem}

\begin{lemma}Let $\zeta$ be a non-trivial $p$-th root of unity, $p$ is prime. Then
$$\sum\limits_{i=1}^{p - 1}\zeta^{i} = -1$$ is the only linear relation over $\mathbb{Q}$. \end{lemma}

\begin{proof} 
Suppose there are two relations: $-1 = r(\zeta) = s(\zeta)$. Let $f(x)$ be a minimal polynomial of an element $\zeta$ over $\mathbb{Q}$. Let us divide $r(x)$ and $s(x)$ by $f(x)$ with remainder. We get $r'(\zeta) =s'(\zeta) =-1$, where $\deg \, r'(x) < \deg \, f(x)$ and $\deg \, s'(x) < \deg \, f(x)$. According to minimal polynomial property $f(x) $ divides $r'(x) - s'(x)$, but then degree of $r'(x) - s'(x)$ is smaller than degree of  $f(x)$; thus, $r(\zeta) = s(\zeta).$
\end{proof}

\begin{cor}\label{cor1}Let $\zeta$ be a non-trivial $p$-th root of unity. If $p$ is prime and 
$$k_{1}\zeta + \ldots + k_{p - 1}\zeta^{p - 1} \in \mathbb{Z},$$ then $k_{1} = \ldots = k_{p - 1}.$\end{cor}

\begin{theorem}[\protect{\cite[Theorems 1, 2, 3, 4]{N}}]\label{max_ord} Let $\sigma$ be an automorphism of a compact Riemann surface $C$ of genus $g \geq 2$. 
\begin{enumerate}{
\item{The Riemann surface having automorphism of maximum order $4g + 2$ is isomorphic to the Riemann surface defined by
$$y^{2} = x(x^{2g + 1} - 1),$$
the automorphism is given by:
 $$\sigma(x, y) = (e^{\frac{2\pi\textbf{i}}{2g + 1}}x, e^{\frac{2\pi\textbf{i}}{4g + 2}}y).$$}

\item{If $g \geq 4$, then the maximum order of an automorphism, which is smaller than $4g + 2$, is equal to $4g$, and the Riemann surface and the automorphism are given by
$$ y^{2} = x(x^{2g} - 1), \, \, \, \sigma(x, y) = (e^{\frac{2\pi\textbf{i}}{2g}}x, e^{\frac{2\pi\textbf{i}}{4g}}y).$$

}

\item{
If $g = 3$, then the maximum order of an automorphism, which is smaller than $15$,  is equal to $9$, and the Riemann surface and the automorphism are given by
$$y^{3} = x(x^{3} - 1),  \, \, \, \sigma(x, y) = (e^{\frac{2\pi\textbf{i}}{3}}x, e^{\frac{2\pi\textbf{i}}{9}}y).$$
}

\item{

If $g \equiv 1$ (mod $3$),  then the maximum order of an automorphism, which is  smaller than  $4g$,  is equal to $3g+3$, and the Riemann surface and the automorphism are given by
$$ y^{3} = x(x^{g + 1} - 1), \, \, \, \sigma(x, y) = (e^{\frac{2\pi\textbf{i}}{g + 1}}x, e^{\frac{2\pi\textbf{i}}{3g + 3}}y).$$

If $g = 4$, then the maximum order of an automorphism, which is smaller than $15$,  is equal to $12$.

}

}\end{enumerate}
\end{theorem} 

\section{Curves of large genus}\label{big}

In this section we will prove Theorem \ref{m1}.

\begin{lemma}\label{ml1} If $x > 0$, then 
	$$\mathrm{cos} \, x \geq 1 - \frac{10x}{13}.$$
\end{lemma}

\begin{proof}

In case when $x \geq \frac{13}{5}$ inequality is correct because  $$\mathrm{cos}\,{x} \geq -1 \geq 1 - \frac{10}{13}\cdot\frac{13}{5}.$$
Let us show that inequality is correct on  $[0, \frac{13}{5}]$. For this purpose we will find extremums of the function
 $$f(x) = \mathrm{cos} \, x - 1 + \frac{10x}{13} .$$ For this let us find roots of the equation $$f'(x) = -\mathrm{sin}\, x +  \frac{10}{13}.$$
The extremums of $f(x)$ are the following
 $$\left[ 
      \begin{gathered} 
        x = \mathrm{arcsin\Bigl{(}\frac{10}{13}\Bigr{)}} + 2\pi k; \\ 
        x = \pi - \mathrm{arcsin\Bigl{(}\frac{10}{13}\Bigr{)}} + 2\pi k, \, k \in \mathbb{Z}.
      \end{gathered} 
\right.$$
There are only two values on $[0, \frac{13}{5}]$: namely, $x = \mathrm{arcsin\bigl{(}\frac{10}{13}\bigr{)}}$ and $ x = \pi - \mathrm{arcsin\bigl{(}\frac{10}{13}\bigr{)}}$. Let us substitute these values in $f(x).$
\begin{multline*}
f\Bigl{(}\mathrm{arcsin\Bigl{(}\frac{10}{13}\Bigr{)}}\Bigr{)} = \mathrm{cos} \Bigl{(} \mathrm{arcsin\Bigl{(}\frac{10}{13}\Bigl{)}}\Bigl{)} - 1 + \frac{10\mathrm{arcsin\bigl{(}\frac{10}{13}\bigr{)}}}{13} = \\
= \sqrt{1 - \Bigl{(}\frac{10}{13}\Bigr{)}^{2}} - 1 + \frac{\mathrm{10arcsin\bigl{(}\frac{10}{13}\bigl{)}}}{13}
= \frac{\sqrt{69}}{13} - 1 +  \frac{10\mathrm{arcsin\bigl{(}\frac{10}{13}\bigr{)}}}{13}.
\end{multline*}

Let us multiply $f(x)$ by 13
$$\sqrt{69} - 13 +  10 \cdot \mathrm{arcsin\Bigl{(}\frac{10}{13}\Bigr{)}} \geq 8 - 13 +10 \cdot \frac{\pi}{4} \geq -5 + \frac{5}{2} \cdot \pi \geq 0.$$

Now substitute the value $\pi - \mathrm{arcsin(\frac{10}{13})}$ in $f(x)$
\begin{multline*}
f\Bigl{(}\pi - \mathrm{arcsin\Bigl{(}\frac{10}{13}\Bigr{)}}\Bigr{)} = \\
= \mathrm{cos} \Bigl{(}\pi - \mathrm{arcsin\Bigl{(}\frac{10}{13}\Bigr{)}}\Bigr{)} - 1 + \frac{10\bigl{(}\pi - \mathrm{arcsin}\bigl{(}\frac{10}{13}\bigr{)}\bigr{)}}{13} = \\
= -\frac{\sqrt{69}}{13} - 1 + 10 \cdot \frac{\pi - \mathrm{arcsin\bigl{(}\frac{10}{13}\bigr{)}}}{13}.
\end{multline*}

Let us multiply $f(x)$ by  13

\begin{multline*}
 -\sqrt{69} - 13 +  10 \cdot\Bigl{(}\pi - \mathrm{arcsin\Bigl{(}\frac{10}{13}\Bigr{)}}\Bigr{)} \geq \\
\geq - \frac{17}{2} - 13 + 10 \cdot \pi - 10 \cdot \mathrm{arcsin\Bigl{(}sin\Bigl{(}\frac{\pi}{3} - \frac{\pi}{24}\Bigr{)}\Bigr{)}} = \\
= -\frac{43}{2} + \frac{170 \pi}{24} \geq 0. 
\end{multline*}

We need to show that $\mathrm{sin\Bigl{(}\frac{\pi}{3} - \frac{\pi}{24}\Bigr{)}} \geq \frac{10}{13}$.

Indeed
\begin{multline*}
\mathrm{sin\Bigl{(}\frac{\pi}{3} - \frac{\pi}{24}\Bigr{)}} = \mathrm{sin\Bigl{(}\frac{\pi}{3}\Bigr{)}cos\Bigl{(}\frac{\pi}{24}\Bigr{)} - cos\Bigl{(}\frac{\pi}{3}\Bigr{)}sin\Bigl{(}\frac{\pi}{24}\Bigr{)}} = \\
 = \mathrm{\frac{\sqrt{3}}{2}cos\Bigl{(}\frac{\pi}{24}\Bigr{)} - \frac{1}{2}sin\Bigl{(}\frac{\pi}{24}\Bigr{)}} = 
\mathrm{\frac{\sqrt{3}}{2} \sqrt{\frac{1 + cos(\frac{\pi}{12})}{2}}} - \mathrm{\frac{1}{2} \sqrt{\frac{1 - cos(\frac{\pi}{12})}{2}}} =\\
= \mathrm{\frac{\sqrt{3}}{2} \sqrt{\frac{1 + \sqrt{\frac{1 + cos(\frac{\pi}{6})}{2}}}{2}}} - \mathrm{\frac{1}{2} \sqrt{\frac{1 - \sqrt{\frac{1 + cos(\frac{\pi}{6})}{2}}}{2}}} = \mathrm{\frac{\sqrt{3}}{4} \sqrt{2 + \sqrt{2 + \sqrt{3}}}} - \mathrm{\frac{1}{4} \sqrt{2 - \sqrt{2 + \sqrt{3}}}} \geq \\
\geq \frac{\sqrt{3}}{4} \sqrt{2 + \sqrt{\frac{37}{10}}} - \frac{1}{4} \sqrt{2 - \sqrt{\frac{37}{10}}} \geq  \frac{\sqrt{3}}{4} \sqrt{2 + \frac{19}{10}} - \frac{1}{4} \sqrt{2 - \frac{19}{10}} = \\
 = \frac{\sqrt{3}}{4} \sqrt{\frac{39}{10}} - \frac{1}{4} \sqrt{\frac{1}{10}} \geq \frac{1}{4}\sqrt{\frac{117}{10}} - \frac{1}{4}\sqrt{\frac{1}{10}} \geq \frac{17}{20} - \frac{2}{25} = \frac{77}{100} > \frac{10}{13}.
\end{multline*}

We have shown that in both extremums $f(x) > 0;$ thus, on the whole $[0, \frac{13}{5}]$ function is positive.\\
\end{proof}

Let us prove Theorem \ref{m1}. The proof is following the proof of Theorem \ref{b1}.

\begin{proof}
According to \ref{kl1} to show that $J/G$ is not uniruled we need to show that action $G$ satisfies the global Reid condition. Let us fix an element $\sigma$ of order $r$ of the group $G$ and point $p$ of the variety $J$ and show that $(\sigma, T_{p}(J))$ satisfies the local Reid condition. Action $G$ on  $ T_{p}(J)$ is isomorphic to the action on  $T_{0}(J) = H^{0, 1}(C)$. The eigenvalues of  $\sigma$ are $\zeta^{a_{1}}, \ldots, \zeta^{a_{g}}$, where $\zeta = e^{\frac{2\pi \mathbf{i}}{r}}$ and $0 \leq a_{i}  < r$. It remains to show that $\sum a_{i} \geq r$ for all $\sigma$ in $G$.

The eigenvalues of $\sigma$ acting on $H^{1}(C, \mathbb{C})$ are  $\zeta^{a_{1}}, \ldots, \zeta^{a_{g}}$,$\zeta^{-a_{1}}, \ldots, \zeta^{-a_{g}}$. Thus, trace of $\sigma$ is equal to  $2\sum \mathrm{cos} \frac{2\pi a_{i}}{r}$. According to the Lefschetz formula the trace is equal to $2 - \mathrm{Fix}_{C}(\sigma)$. Thus, the trace is equal or less than 2. Let us use that  $x \geq 1 - \frac{10x}{13}$ according to  \ref{ml1}. Then
$$1 \geq \sum \mathrm{cos} \frac{2\pi a_{i}}{r} \geq \sum \Bigl{(}1 - \frac{20\pi a_{i}}{13r}\Bigr{)} \geq g - \frac{20\pi}{13r}\sum a_{i}.$$

Hence, $\sum a_{i} \geq \frac{g - 1}{20\pi}13r$. Thus, if $$g \geq 1 + \frac{20\pi}{13} = 5.8332\ldots,$$ we get $\sum a_{i} > r$.
In the case when $g = 5$ the statement follows from Theorem \ref{b2}.\\
\end{proof}

\section{Curves of genus 4}\label{four}

 In this section we will prove Theorem \ref{m2}.

{We will write down all the necessary conditions for the existence of an automorphism and that quotient is not uniruled. Namely, Lefschetz formula for different degrees of automorphism and opposite inequality to the Reid condition. In case of existence of the remaining automorphisms, the quotient will not be uniruled. The next case is also possible: there is an automorphism for which the Reid condition is met but not for its degree. In this case we need to check those degrees of automorphism that are coprime with its order. The reason is that when the order is not coprime, we get an automorphism that we have already checked.}

\begin{lemma}\label{4m1}
Let $C$ be a compact Riemann surface of genus g = 4, and $\sigma$ is an automorphism of order~3. Then $J/\langle\sigma\rangle$ is not uniruled.
\end{lemma}

\begin{proof}

Let us denote $\omega = e^{\frac{2\pi \mathbf{i}}{3}}$. Let the eigenvalues of  $\sigma$ be
\[
\underbracket{1, \ldots, 1}_{k}, \, \underbracket{\omega, \ldots, \omega}_{l}, \, \underbracket{\omega^{2}, \ldots, \omega^{2}}_{4 - k - l}.
\]
Then from the Lefschetz formula it follows that
\begin{multline*}
k + l\omega + (4 - k - l)\omega^{2} + k + l\omega^{2} + (4 - k - l)\omega = 2k + (4 - k)\omega + (4 - k)\omega^{2} = 3k - 4 \leq 2. 
\end{multline*}

We get that $k \leq 2.$ 
If $k = 2$ then the global Reid condition has the following form:
$$l + 2(2 - l) = 4 - l \geq 3.$$
Thus, we get that only when $l \geq 2$ the condition is met. However, since $k + l \leq 4$ we get that $l = 2$. Then eigenvalues of  $\sigma$ are $1, 1, \omega, \omega.$
According to \ref{example} we know that this case does not exist.

If $k = 1$ then the global Reid condition has the following form:
$$l + 2(3 - l) = 6 - l \geq 3.$$
We get that if $l > 4$ then the condition is not met. But $k + l \leq 4$; thus, the condition is always met.

If $k = 0$ then the global Reid condition has the following form:
$$l + 2(4 - l) = 8 - l \geq 3.$$
The condition is not met only if $l > 5$. 
\end{proof}

\begin{lemma}\label{4m2}
Let $C$ be a compact Riemann surface of genus g = 4, and $\sigma$ is an automorphism of prime order. Then $J/\langle\sigma\rangle$ is not uniruled.
\end{lemma}

\begin{proof}
According to \ref{order} the maximum order of an automorphism is equal or smaller than $4g + 2$.  That is why it is sufficient to consider automorphisms of order 18 or smaller.

Let us consider the case when $N = 2$.
If the eigenvalues of  $\sigma$ are 
Then from the Lefschetz formula it follows that

$$k + (4 - k)(-1) + k + (4 - k)(-1) = 4k - 8 \leq 2 $$
We get that $k \leq 2$.

For any $k = 0, 1, 2$ the global Reid condition is correct. Thus $J/G$ is not uniruled in case of action of group of order $2$.

The case when $N = 3$ is considered in lemma \ref{4m1}.

Let us suppose that automorphism order is $N \geq 5$. Let $\zeta^{i}$, where $i = 0, \ldots, N - 1$ the eigenvalues and the quantity of them is equal to $k_{i}$, $i = 0, \ldots, N - 1$ respectively. 

There are eight eigenvalues of $\sigma$ acting on $H^{1}(C, \mathbb{C})$. If there is a $\zeta^{a}$ amid them then there is also $\zeta^{-a}.$ According to \ref{cor1} all coefficients are required to be equal for sum to be integer.
Then if $N = 5$, there are following possible eigenvalues: $\zeta$, $\zeta^{2}$, $\zeta^{3}$, $\zeta^{4}$.
In case when $N = 7$, the following: 1, $\zeta^{a}$, $\zeta^{b}$, $\zeta^{c}$, where $a, b, c$ are different.
In cases when $N = 11$, $N = 13$ and $N = 17$ there are no possible eigenvalues.

if $N = 5$ then $\zeta$, $\zeta^{2}$, $\zeta^{3}$, $\zeta^{4}$ satisfy the global Reid condition, thus $J/G$ is not uniruled.

If $N = 7$ then 1, $\zeta^{a}$, $\zeta^{b}$, $\zeta^{c}$, where $a, b, c$ are different, do not satisfy the global Reid condition only if $a = 1$, $b = 2$, $c = 3$. Then from the Lefschetz formula it follows that this kind of action has only one fixed point. However, according to \ref{fk}, if the action of prime order has a fixed point it has at least two fixed points. That is why this case is also not possible.

We get that in case of prime order of an automorphism the variety $J/\langle\sigma\rangle$  is not uniruled.
\end{proof}

\begin{lemma}\label{4m3}
Let $C$ be a compact Riemann surface of genus $g = 4$, and $\sigma$ is an automorphism of order~4. Then $J/\langle\sigma\rangle$ is not uniruled.
\end{lemma}

\begin{proof}
Let the quantity of the eigenvalues $1, \mathbf{i}, -1, -\mathbf{i}$ of $\sigma$ be equal to $k$, $l$, $m$, $4 - k - l - m$ respectively.

From the Lefschetz formula it follows that:
\begin{multline*}
k + l\mathbf{i} + m(-1) + (4 - k - l - m)(-\mathbf{i}) + k + l(-\mathbf{i}) + m(-1) + (4 - k - l - m)\mathbf{i} = \\
= 2k - 2m + (4 - k - m)\mathbf{i} + (4 - k  - m)(-\mathbf{i}) = 2(k - m) \leq 2.
\end{multline*}

Thus, $k \leq m + 1$ and $k + m \leq 4$.

If the automorphism  $\sigma$ exists, then all its degrees exist too. The eigenvalues of $\sigma^{2}$ are $1$ and $-1$ and there are $k + m$ and $4 - k - m$ of them respectively.
From the Lefschetz formula it follows that:
$$2(k + m) - 2(4 - k - m) \geq 2.$$

The global Reid condition is
$$l + 2m + 3(4 - k - l - m) \geq 4.$$

Let us write down all possible sets of numbers  $k$, $l$, $m$, $4 - k - l - m$ which do not satisfy the global Reid condition and check which of them satisfy the Lefschetz formula. In the fifth column by + we mark sets which satisfy the Lefschetz formula.
\begin{table}[h]
  \begin{center}
    \caption{$N = 4$}
    \label{tab:table1}

    \begin{tabular}{|c|c|c|c|c|c|c|} 
      \hline
      $k$ & $l$ & $m$ & $4 - k - l - m$ & Lefschetz \\
      \hline
      3 & 0 & 0 & $1$ & $-$ \\
      \hline
      2 & 1 & 1 & 0 & $+$  \\
      \hline
      3 & 0 & 1 & 0 & $-$ \\
      \hline
      1 & 3 & 0 & 0 & $+$  \\
      \hline
      2 & 2 & 0 & 0 & $-$ \\
      \hline
      3 & 1 & 0 & 0 & $-$ \\
      \hline
      4 & 0 & 0 & 0 & $-$ \\
      \hline
    \end{tabular}
  \end{center}
\end{table}

We get two possible automorphisms with eigenvalues $1, \textbf{i}, \textbf{i}, \textbf{i}$ and $1, 1, \textbf{i}, -1$. Let us check whether the condition $\widetilde{E}$ is met for them.
If the eigenvalues are $1, \textbf{i}, \textbf{i}, \textbf{i}$ then the condition $\widetilde{E}$ has the following form:
$$ 1 + 3\textbf{i} = 1 + f_{1}\frac{\textbf{i}}{1 - \textbf{i}}+ f_{3}\frac{-\textbf{i}}{1 + \textbf{i}}.$$
Modifying it we get:
$$6\textbf{i} = \textbf{i}(f_{1} - f_{3}) - (f_{1} + f_{3}).$$
We get that $f_{1} = 3$, and $f_{3} = -3$ hence the condition is not met.
If the eigenvalues are $1, 1, \textbf{i}, -1$  then the condition $\widetilde{E}$ has the following form:
$$2\textbf{i} = \textbf{i}(f_{1} - f_{3}) - (f_{1} + f_{3}).$$
We get that $f_{1} = 1,$ and $f_{3} = -1$ hence the condition is not met.
\end{proof}

\begin{lemma}\label{4m4}
Let $C$ be a compact Riemann surface of genus $g = 4$, and $\sigma$ is an automorphism of order~6. Then $J/\langle\sigma\rangle$ is not uniruled.
\end{lemma}

\begin{proof}

Let $\omega = e^{\frac{2\pi \mathbf{i}}{3}}$ and the quantity of the eigenvalues $1, -\omega, \omega^{2}, -1, \omega, -\omega^{2}$ of $\sigma$ is equal to $k_{0}$, $k_{1}$, $k_{2}$, $k_{3}$, $k_{4}$, $k_{5}$ respectively.
Let us write down inequalities following from the Lefschetz formula for automorphisms $\sigma, \sigma^{2}, \sigma^{3}$ and inequality for sets which do not satisfy the global Reid condition:
\begin{equation*}
 \begin{cases}
2k_{0} + k_{1} + k_{2}(-1) + k_{3}(-2) + k_{4}(-1) + k_{5}(-1) \leq 2; \\
2(k_{0} + k_{3}) - (k_{1} + k_{4}) - (k_{2} + k_{5}) \leq 2; \\
2(k_{0} + k_{2} + k_{4}) - 2(k_{1} + k_{3} + k_{5}) \leq 2; \\
k_{1} + 2k_{2} + 3k_{3} + 4k_{4} + 5k_{5} < 6.
 \end{cases}
\end{equation*}

Let us write down all possible sets of numbers $k_{0}$, $k_{1}$, $k_{2}$, $k_{3}$, $k_{4}$, $k_{5}$ which do not satisfy the global Reid condition and satisfy all three Lefschetz inequalities:

  \begin{center}
    \begin{minipage}{0.25\textwidth}
      \begin{enumerate}
        \item[1)]{0, 3, 1, 0, 0, 0;}
        \item[2)]{1, 2, 0, 1, 0, 0.}
      \end{enumerate}
    \end{minipage}
  \end{center}

We get two possible sets of eigenvalues

  \begin{center}
    \begin{minipage}{0.25\textwidth}
      \begin{enumerate}
        \item[1)]{1, $-\omega$, $-\omega$, $-1$;}
        \item[2)]{$-\omega$, $-\omega$, $-\omega$, $\omega^{2}$.}
      \end{enumerate}
    \end{minipage}
  \end{center}

If the automorphism $\sigma$ exists then $\sigma^{2}$ exists. According to \ref{example} squared first automorphism does not exist. For the second automorphism let us write down the Riemann-Hurwitz formula. The eigenvalues of $\sigma^{2}$ are $\omega^{2}, \omega^{2}, \omega^{2}, \omega$. Since $\sigma^{2}$ has prime order then quantity of ramification points is equal to quantity of fixed points. Let us find out the genus of $C/\langle \sigma^{2} \rangle $:
$$2(4 - 1) = 3 (g(C/\langle \sigma^{2} \rangle) - 1) + 6 (3 - 1). $$
We get that genus of $C/\langle \sigma^{2} \rangle $ is negative, hence automorphism does not exist.
\end{proof}

\begin{lemma}\label{4m5}
Let $C$ be a compact Riemann surface of genus $g = 4$, and $\sigma$ is an automorphism of order~8. Then $J/\langle\sigma\rangle$ is not uniruled.
\end{lemma}

\begin{proof}

Let $1, \zeta, \mathbf{i}, \zeta^{3} , -1, \zeta^{5}, -\mathbf{i}, \zeta^{7}$ be the eigenvalues of  $\sigma$, where $\zeta$ is a primitive root modulo 8.
The quantities of eigenvalues are $ k_{0}$, $k_{1}$, $k_{2}$, $k_{3}$, $k_{4}$, $k_{5}$, $k_{6}$, $k_{7}$ respectively. Let us write down inequalities that the sets must satisfy:

\begin{equation*}
 \begin{cases}
2(k_{0} - k_{4}) + \sqrt{2}(k_{1} + k_{7} - k_{3} - k_{5}) \leq 2; \\
2(k_{0} + k_{4}) - 2(k_{2} + k_{6}) \leq 2; \\
2(k_{0} + k_{2} + k_{4} + k_{6}) - 2(k_{1} + k_{3} + k_{5} + k_{7}) \leq 2; \\
k_{1} + 2k_{2} + 3k_{3} + 4k_{4} + 5k_{5}  +  6k_{6} +  7k_{7} < 8.
 \end{cases}
\end{equation*}

Let us write down all possible sets $k_{0}$, $k_{1}$, $k_{2}$, $k_{3}$, $k_{4}$, $k_{5}$, $k_{6}$, $k_{7}$ that do not satisfy the global Reid condition and satisfy the Lefschetz formula:

  \begin{center}
    \begin{minipage}{0.25\textwidth}
      \begin{enumerate}
        \item[1)]{1, 1, 1, 1, 0, 0, 0, 0;}
        \item[2)]{0, 2, 1, 1, 0, 0, 0, 0;}
        \item[3)]{1, 1, 0, 2, 0, 0, 0, 0.}
      \end{enumerate}
    \end{minipage}
  \end{center}

In case 2 and 3 we get that by the Lefschetz formula number of fixed points is not integer, and hence these automorphisms do not exist.
We get the only one automorphism of order 8 that satisfy our conditions. This is the automorphism with eigenvalues $1, \zeta, \textbf{i}, \zeta^{3}.$ Let us check whether the condition $\widetilde{E}$ is met:

$$1 + \zeta + \textbf{i} + \zeta^{3} = 1 + f_{1}\frac{\zeta}{1 - \zeta} + f_{3}\frac{\zeta^{3}}{1 - \zeta^{3}} + f_{5}\frac{\zeta^{5}}{1 - \zeta^{5}}+ f_{7}\frac{\zeta^{7}}{1 - \zeta^{7}}.$$
Modifying this expression we get
$$(2 + \sqrt{2})f_{1} + (1 + \textbf{i})f_{3} + \textbf{i}(\sqrt{2}f_{5} + (1 + \textbf{i})f_{7}) + (1 + 2\textbf{i})\sqrt{2} + (1 + 3\textbf{i})) = 0.$$
Let us multiply the equation by $1 - \frac{1}{\sqrt{2}} + \frac{\textbf{i}}{\sqrt{2}}$:
$$\textbf{i}(\sqrt{2} + 1 - \textbf{i})f_{1} + (-\sqrt{2} +1 + \textbf{i})f_{3} + \textbf{i}(\sqrt{2} -1 + \textbf{i})f_{5} + (-\sqrt{2} -1 - \textbf{i})f_{7} = (\sqrt{2} + 2)(1 + \textbf{i}).$$
We get that there are no such $f_{1}, f_{3}, f_{5}, f_{7}$ that satisfy the condition $\widetilde{E}$. Hence there is no such an automorphism of order 8, satisfying our conditions. 
\end{proof}

\begin{lemma} \label{4m6}
Let $C$ be a compact Riemann surface of genus $g = 4$, and $\sigma$ is an automorphism or order~9. Then $J/\langle\sigma\rangle$ is not uniruled.
\end{lemma}

\begin{proof}

Let $\zeta = e^{\frac{2\pi \mathbf{i}}{9}}$, $\omega = e^{\frac{2\pi \mathbf{i}}{3}}$, and the quantity of eigenvalues 1, $\zeta$, $\zeta^{2}$, $\omega$, $\zeta^{4}$, $\zeta^{5}$, $\omega^{2}$, $\zeta^{7}$, $\zeta^{8}$ is equal to  $ k_{0}$, $k_{1}$, $k_{2}$, $k_{3}$, $k_{4}$, $k_{5}$, $k_{6}$, $k_{7}$, $k_{8}$ respectively.
Let us write down the inequalities which follow from the Lefschetz formula for $\sigma$ and $\sigma^3$:

\begin{equation*}
 \begin{cases}
 2k_{0} + 2(k_{1} + k_{8})\mathrm{cos}\Bigl{(}\frac{2\pi}{9}\Bigr{)} + 2(k_{2} + k_{7})\mathrm{cos}\Bigl{(}\frac{4\pi}{9}\Bigr{)} - (k_{3} + k_{6}) + 2(k_{4} + k_{5})\mathrm{cos}\Bigl{(}\frac{8\pi}{9}\Bigr{)} \leq 2;\\
2(k_{0} + k_{3} + k_{6}) - (k_{1} + k_{4} + k_{7}) - (k_{2} + k_{5} + k_{8}) \leq 2.
 \end{cases}
\end{equation*}

Let us write down all possible sets $k_{0}$, $k_{1}$, $k_{2}$, $k_{3}$, $k_{4}$, $k_{5}$, $k_{6}$, $k_{7}$,  $k_{8}$, that do not satisfy the global Reid condition and satisfy the Lefschetz formula. Let us denote by |$\mathrm{Fix}_{C}(\sigma)$| the quantity of fixed points of automorphism $\sigma$.
\begin{table}[H]
\begin{center}
\caption{$N = 9$}
\label{tab:table2}
\begin{tabular}{|c|c|c|c|c|c|c|c|c|c|} 
\hline
$k_{0}$ & $k_{1}$ & $k_{2}$ & $k_{3}$ & $k_{4}$ & $k_{5}$ & $k_{6}$ & $k_{7}$ & $k_{8}$ & |$\mathrm{Fix}_{C}(\sigma)$| \\
\hline
0 & 0 & 4 & 0 & 0 & 0 & 0 & 0 & 0 & — \\
\hline
1 & 0 & 2 & 1 & 0 & 0 & 0 & 0 & 0 & — \\
\hline
0 & 1 & 2 & 1 & 0 & 0 & 0 & 0 & 0 & — \\
\hline
0 & 2 & 0 & 2 & 0 & 0 & 0 & 0 & 0 & — \\
\hline
1 & 1 & 1 & 0 & 1 & 0 & 0 & 0 & 0 & 0 \\
\hline
0 & 2 & 1 & 0 & 1 & 0 & 0 & 0 & 0 & — \\
\hline
1 & 0 & 2 & 0 & 1 & 0 & 0 & 0 & 0 & — \\
\hline
1 & 1 & 0 & 1 & 1 & 0 & 0 & 0 & 0 & — \\
\hline
2 & 0 & 0 & 0 & 2 & 0 & 0 & 0 & 0 & — \\
\hline
1 & 1 & 1 & 0 & 0 & 1 & 0 & 0 & 0 & 0 \\
\hline
\end{tabular}
\end{center}
\end{table}

There are only two possible sets if the eigenvalues: 1, $\zeta$, $\zeta^{2}$, $\zeta^{5}$  and 1, $\zeta$, $\zeta^{2}$, $\zeta^{4}$. Both of them do not satisfy the condition $\widetilde{E}$.
\end{proof}

\begin{lemma}\label{4m7}
Let $C$ be a compact Riemann surface of genus $g = 4$, and $\sigma$ is an automorphism or order~10. Then $J/\langle\sigma\rangle$ is not uniruled.
\end{lemma}

\begin{proof}

Let $1, \zeta,  \zeta^{2}, \zeta^{3} ,  \zeta^{4}, \zeta^{5},  \zeta^{6}, \zeta^{7},  \zeta^{8},  \zeta^{9}$ be the eigenvalues of  $\sigma$, where $\zeta$  is a primitive root modulo 10. The quantity of eigenvalues is equal to $ k_{0}$, $k_{1}$, $k_{2}$, $k_{3}$, $k_{4}$, $k_{5}$, $k_{6}$, $k_{7}$, $k_{8}$, $k_{9}$ respectively.

Let us write down the inequalities which follow from the Lefschetz formula for $\sigma$, $\sigma^2$ and $\sigma^5$  and inequality for sets which do not satisfy the global Reid condition:
\begin{equation*}
 \begin{cases}
2k_{0} + 2(k_{1} + k_{9})\mathrm{cos}\Bigl{(}\frac{2\pi}{10}\Bigr{)} + 2(k_{2} + k_{8})\mathrm{cos}\Bigl{(}\frac{4\pi}{10}\Bigr{)}  + 2(k_{3} + k_{7})\mathrm{cos}\Bigl{(}\frac{6\pi}{10}\Bigr{)} + 2(k_{4} + k_{6})\mathrm{cos}\Bigl{(}\frac{8\pi}{10}\Bigr{)} - 2k_{5} \leq 2; \\
(k_{0} + k_{5}) + \mathrm{cos}\Bigl{(}\frac{2\pi}{5}\Bigr{)}(k_{1} + k_{4} + k_{6} + k_{9}) + \mathrm{cos}\Bigl{(}\frac{4\pi}{5}\Bigr{)}(k_{2} + k_{3} + k_{7} + k_{8}) \leq 1; \\
k_{0} + k_{2} + k_{4} + k_{6} + k_{8} - k_{1} - k_{3} - k_{5} - k_{7} - k_{9} \leq 1; \\
k_{1} + 2k_{2} + 3k_{3} + 4k_{4} + 5k_{5}  +  6k_{6} +  7k_{7} +  8k_{8}  + 9k_{9} < 10.
 \end{cases}
\end{equation*}

Let us write down all possible sets $k_{i}$, $i = 0, \ldots, 9$ that do not satisfy the global Reid condition and satisfy the Lefschetz formula.

\begin{table}[h]
  \begin{center}
   \caption{$N = 10$}
   \label{tab:table3}
 \begin{tabular}{|c|c|c|c|c|c|c|c|c|c|c|} 
      \hline
      $k_{0}$ & $k_{1}$ & $k_{2}$ & $k_{3}$ & $k_{4}$ & $k_{5}$ & $k_{6}$ & $k_{7}$ & $k_{8}$ &  $k_{9}$ & |$\mathrm{Fix}_{C}(\sigma)$| \\
      \hline
0 & 1 & 1 & 2 & 0 & 0 & 0 & 0 & 0 & 0 & 1\\
\hline
0 & 2 & 0 & 1 & 1 & 0 & 0 & 0 & 0 & 0 & 1\\
\hline
1 & 0 & 0 & 3 & 0 & 0 & 0 & 0 & 0 & 0 & $-$\\
\hline
1 & 0 & 1 & 2 & 0 & 0 & 0 & 0 & 0 & 0 & $-$\\
\hline
1 & 1 & 0 & 1 & 1 & 0 & 0 & 0 & 0 & 0 & $-$\\
\hline
    \end{tabular}
  \end{center}
\end{table}

We get two possible sets:

  \begin{center}
    \begin{minipage}{0.3\textwidth}
      \begin{enumerate}
        \item[1)]{0, 1, 1, 2, 0, 0, 0, 0, 0, 0;}
        \item[2)]{0, 2, 0, 1, 1, 0, 0, 0, 0, 0.}
      \end{enumerate}
    \end{minipage}
  \end{center}

Both of them does not exist because the squared automorphisms have non-integer amount of fixed points.
\end{proof}

\begin{lemma}\label{4m8}
Let $C$ be a compact Riemann surface of genus $g = 4$, and $\sigma$ is an automorphism or order~12. Then $J/\langle\sigma\rangle$ is not uniruled.
\end{lemma}

\begin{proof}

Let $1, \zeta,  -\omega, \mathbf{i} ,  \omega^{2}, \zeta^{5},  -1, \zeta^{7},  \omega,  -\mathbf{i}, -\omega^{2}, \zeta^{11} $be the eigenvalues of  $\sigma$, where $\zeta$is a primitive root modulo 12. The quantity of eigenvalues is equal to  $ k_{0}$, $k_{1}$, $k_{2}$, $k_{3}$, $k_{4}$, $k_{5}$, $k_{6}$, $k_{7}$, $k_{8}$, $k_{9}$, $k_{10}$, $k_{11}$  respectively.

Let us write down inequalities that the sets must satisfy: 
\begin{equation*}
 \begin{cases}
2k_{0} + \sqrt{3}(k_{1} + k_{11} - k_{5} - k_{7})+ k_{2} - k_{4} - 2k_{6} - k_{8} + k_{10} \leq 2; \\
2(k_{0} + k_{6}) + (k_{1} + k_{4} + k_{5} + k_{7}) - (k_{2} + k_{4} + k_{8} + k_{10}) - 2(k_{3} + k_{9}) \leq 2; \\
2(k_{0} + k_{3} + k_{6} + k_{9}) - k_{1} - k_{4} - k_{7} - k_{10} - k_{2} - k_{5} - k_{8} - k_{11} \leq 2;\\
k_{0} + k_{2} + k_{4} + k_{6} + k_{8} + k_{10} - k_{1} - k_{3} - k_{5} - k_{7} - k_{9} - k_{11} \leq 1; \\
k_{1} + 2k_{2} + 3k_{3} + 4k_{4} + 5k_{5}  +  6k_{6} +  7k_{7} +  8k_{8} + 9k_{9} + 10k_{10}+ 11k_{11} < 12.
 \end{cases}
\end{equation*}

Let us write down all possible sets of numbers $k_{i}$, $i = 0, \ldots, 11$, which satisfy these inequalities:

\begin{table}[h]
  \begin{center}
    \caption{$N = 12$}
    \label{tab:table4}
    \begin{tabular}{|c|c|c|c|c|c|c|c|c|c|c|c|c|} 
      \hline
      $k_{0}$ & $k_{1}$ & $k_{2}$ & $k_{3}$ & $k_{4}$ & $k_{5}$ & $k_{6}$ & $k_{7}$ & $k_{8}$ &  $k_{9}$ &  $k_{10}$ & $k_{11}$ & |$\mathrm{Fix}_{C}(\sigma)$|  \\
      \hline
0 & 0 & 2 & 2 & 0 & 0 & 0 & 0 & 0 & 0 & 0 & 0 & 0\\
\hline
0 & 1 & 0 & 2 & 1 & 0 & 0 & 0 & 0 & 0 & 0 & 0 & $-$ \\
\hline
0 & 1 & 1 & 1 & 0 & 1 & 0 & 0 & 0 & 0 & 0 & 0 & 1 \\
\hline
0 & 1 & 1 & 1 & 1 & 0 & 0 & 0 & 0 & 0 & 0 & 0 & $-$ \\
\hline
0 & 1 & 2 & 0 & 0 & 1 & 0 & 0 & 0 & 0 & 0 & 0 & 0\\
\hline
0 & 2 & 0 & 1 & 0 & 0 & 1 & 0 & 0 & 0 & 0 & 0 & $-$ \\
\hline
0 & 2 & 0 & 1 & 0 & 1 & 0 & 0 & 0 & 0 & 0 & 0 & $-$\\
\hline 
1 & 0 & 1 & 1 & 0 & 1 & 0 & 0 & 0 & 0 & 0 & 0 & $-$\\
\hline
1 & 1 & 0 & 1 & 0 & 0 & 0 & 1 & 0 & 0 & 0 & 0 & 0 \\
\hline
1 & 1 & 0 & 1 & 0 & 1 & 0 & 0 & 0 & 0 & 0 & 0 & 0\\
\hline
    \end{tabular}
  \end{center}
\end{table}

We get 5 possible sets:

  \begin{center}
    \begin{minipage}{0.35\textwidth}
      \begin{enumerate}
        \item[1)]{0, 0, 2, 2, 0, 0, 0, 0, 0, 0, 0, 0;}
        \item[2)]{0, 1, 1, 1, 0, 1, 0, 0, 0, 0, 0, 0;}
        \item[3)]{0, 1, 2, 0, 0, 1, 0, 0, 0, 0, 0, 0;}
        \item[4)]{1, 1, 0, 1, 0, 0, 0, 1, 0, 0, 0, 0;}
        \item[5)]{1, 1, 0, 1, 0, 1, 0, 0, 0, 0, 0, 0.}
      \end{enumerate}
    \end{minipage}
  \end{center}

Sets 1 and 4 do not exist since the corresponding automorphisms in degree 4 do not exist according to lemma \ref{example}. Set 3 does not exist since its sixth degree does not exist since its amount of fixed points is negative  according to the Lefschetz formula. The second set does not satisfy the condition $\widetilde{E}$. Thus, $J/\langle\sigma\rangle$  is not uniruled.
\end{proof}

Let us prove Theorem \ref{m2}.

\begin{proof}

Let us consider possible orders of $N$. According to \ref{order} value $N$ can vary from 2 to 18. We have already studied the following cases and it has been already shown that there is no automorphisms of order equal or smaller then 12 such that the quotient is uniruled.

$N = 3:$ \, \, lemma $\ref{4m1};$

$N = 2, 5, 7, 11, 13:$ lemma $\ref{4m2};$

$N = 4:$ \, \, lemma $\ref{4m3};$

$N = 6:$ \, \, lemma $\ref{4m4};$

$N = 8:$ \, \, lemma $\ref{4m5};$

$N = 9:$ \, \, lemma $\ref{4m6};$

$N = 10:$ \, lemma $\ref{4m7};$

$N = 12:$ \, lemma $\ref{4m8}.$

Let us consider the remaining cases.\\

Let $N = 14$. According to \ref{max_ord} (paragraph 4) there is no automorphisms of order 14 on a compact Riemann surface of order 4. \\

Let $N = 15$. According to \ref{max_ord} (paragraph 4) there is only one compact Riemann surface of genus~4:
	$$y^{3} = x(x^{5} - 1)$$
and an automorphism of order $15$:
	$$\sigma(x, y) = (\zeta^{3} x, \zeta y).$$
The basis of $H^{0, 1}(C) \cong H^{0}(C, K_{C})$ is $\frac{dx}{y^{2}}, \frac{xdx}{y^{2}}, \frac{x^{2}dx}{y^{2}}, \frac{dx}{y}$. Then the eigenvalues of $\sigma$ are $\zeta, \zeta^{2}, \zeta^{4}, \zeta^{7}.$ They do not satisfy the global Reid condition.

Let $N = 16$. According to \ref{max_ord} (paragraph 2)  there is only one compact Riemann surface of genus~4:
	$$y^{2} = x(x^{8} - 1)$$ 
and an automorphism of order $16$:
	$$\sigma(x, y) = (\zeta^{2} x, \zeta y).$$
The basis of $H^{0}(C, K_{C})$ IS $\frac{dx}{y}, \frac{xdx}{y}, \frac{x^{2}dx}{y}, \frac{x^{3}dx}{y}$.  Then the eigenvalues of $\sigma$ are $\zeta, \zeta^{3}, \zeta^{5}, \zeta^{7}.$ They do not satisfy the global Reid condition.
Thus the quotient of Jacobian is not uniruled.\\

Let $N = 18$. According to \ref{max_ord} (paragraph 1) there is only one compact Riemann surface of genus~4:
	$$y^{2} = x(x^{9} - 1)$$
and an automorphism of order 18:
	$$\sigma(x, y) = (\zeta^{2} x, \zeta y).$$
The basis of $H^{0}(C, K_{C})$ is $\frac{dx}{y}, \frac{xdx}{y}, \frac{x^{2}dx}{y}, \frac{x^{3}dx}{y}$.Then the eigenvalues of $\sigma$ are $\zeta, \zeta^{3}, \zeta^{5}, \zeta^{7}$. The global Reid condition is not met. Hence $J/G$ is uniruled.\\
\end{proof}

\section{Curves of genus 3}\label{three}

{In this section we will prove Theorem \ref{m3}.}

\begin{lemma}\label{5m1}
Let $C$ be a compact Riemann surface of genus $g = 3$, and $\sigma$ is an automorphism of prime order. Then $J/\langle\sigma\rangle$ uniruled if and only if either the order of automorphism is 2 and the eigenvalues of $\sigma$ are $1, 1, -1$ or the order of automorphism is 7 and the eigenvalues of  $\sigma$ are $\zeta, \zeta^{2}, \zeta^{3}$.
\end{lemma}

\begin{proof}
According to \ref{fb} there are three possible automorphisms  of prime order: 2, 3 and 7.

Let $N = 2$. If the eigenvalues of  $\sigma$ are

Then from the Lefschetz formula it follows that
$$k + (3 - k)(-1) + k + (3 - k)(-1) = 4k - 6 \leq 2. $$

We get that $k \leq 2$.

If $k = 0, 1$, the global Reid condition is met.
We get that there is only one possible set of eigenvalues: $1, 1, -1$. This set satisfies the condition $\widetilde{E}$. According to \ref{cr_e_p} such automorphism is realized.
Let $N = 3$. Let $\omega = e^{\frac{2\pi\mathbf{i}}{3}}$ and the quantity of eigenvalues $1, \omega, \omega^{2}$ of $\sigma$ be equal to $k, l, 3 - k - l$. From the Lefschetz formula it follows that 
$$k + l\omega + (3 - k - l)\omega^{2} + k + l\omega^{2} + (3 - k - l)\omega = 2k - l - (3 - k - l) = 3k - 3 \leq 2. $$

We get that $k \leq 1.$ If $k = 1$ then the global Reid condition has the following form
$$l + 2(2 - l) = 4 - l \geq 3.$$ 

The inequality does not hold only if $l = 2, 3.$ Since $k + l \leq 3$ we get that in this case there is only one possible set of eigenvalues: $1, \omega, \omega.$ This automorphism does not satisfy inequalities from \ref{example}.

If  $k = 0$, then the global Reid condition does not hold only if $l \geq 4$; thus, there is no such automorphism.

If $N = 7$, there is one possible set of eigenvalues: $\zeta, \zeta^{2}, \zeta^{3}$. This set satisfies the condition $\widetilde{E}$. This automorphism is realized due to Theorem \ref{cr_e_p}.
\end{proof}

\begin{lemma}\label{5m3}
Let C be a compact Riemann surface of genus g = 3, and $\sigma$ is an automorphism of order~4. Then $J/\langle\sigma\rangle$ is not uniruled.
\end{lemma}

\begin{proof}

Let the quantity of eigenvalues $1, \mathbf{i}, -1, -\mathbf{i}$ of $\sigma$ be equal to $k, l, m,$ $ 3 - k - l - m$ respectively.

Let us write down the inequalities following from the Lefschetz formula for automorphisms $\sigma$ and $\sigma^{2}$ and inequality opposite to the global Reid condition:

\begin{equation*}
 \begin{cases}
k \leq m + 1; \\
k + m \leq 2; \\
l + 2m + 3(3 - k - l - m) < 4. 
 \end{cases}
\end{equation*}

There are three possible sets satisfying the system:

  \begin{center}
    \begin{minipage}{0.25\textwidth}
      \begin{enumerate}
        \item[1)]{0, 3, 0, 0;}
        \item[2)]{1, 1, 1, 0;}
        \item[3)]{1, 2, 0, 0.}
      \end{enumerate}
    \end{minipage}
  \end{center}

None of then satisfies the condition $\widetilde{E}$, thus none of them is realized.

Among the sets which satisfy the Reid condition and which squares are equal to $1, 1, -1$, there is also no possible set. Indeed, the coefficient $r_{\chi, 1}(\sigma^{2})$ from the condition  $\widetilde{E}$ is equal to zero. Thus for the condition $\widetilde{RH}$ to be met, the coefficients $r_{\chi, u}(\sigma)$, $u = 1, 5$ must be equal to zero. There are no such sets.
\end{proof}

\begin{lemma}\label{5m4}
Let C be a compact Riemann surface of genus  g = 3, and $\sigma$ is an automorphism of order~6. Then  $J/\langle\sigma\rangle$ is uniruled if and only if the eigenvalues of action $\sigma$ are equal to $-\omega, \omega^{2}, \omega^{2}$, where $\omega = e^{\frac{2\pi \mathbf{i}}{3}}$.
\end{lemma}

\begin{proof}

Let $\omega = e^{\frac{2\pi \mathbf{i}}{3}}$ and the quantity of eigenvalues $1, -\omega, \omega^{2}, -1, \omega, -\omega^{2}$ of $\sigma$ is equal to $k_{0}$, $k_{1}$, $k_{2}$, $k_{3}$, $k_{4}$, $k_{5}$ respectively.

Let us write down the inequalities following from the Lefschetz formula for automorphisms $\sigma$, $\sigma^{2}$, $\sigma^{3}$ and inequality opposite to the global Reid condition:

\begin{equation*}
 \begin{cases}
$$2k_{0} + k_{1} + k_{2}(-1) + k_{3}(-2) + k_{4}(-1) + k_{5}(-1) \leq 2$$; \\
2(k_{0} + k_{3}) - k_{1} - k_{2} - k_{4} - k_{5} \leq 2; \\
k_{0} + k_{2} + k_{4} - k_{1} - k_{3} - k_{5} \leq 1; \\
k_{1} + 2k_{2} + 3k_{3} + 4k_{4} + 5k_{5} < 6.
 \end{cases}
\end{equation*}

Let us write down all possible variants of sets of numbers $k_{0}$, $k_{1}$, $k_{2}$, $k_{3}$, $k_{4}$, $k_{5}$ which satisfy the system:

\begin{center}
    \begin{minipage}{0.25\textwidth}
      \begin{enumerate}
        \item[1)]{0, 1, 2, 0, 0, 0;}
        \item[2)]{0, 2, 0, 1, 0, 0;}
        \item[3)]{0, 2, 1, 0, 0, 0;}
        \item[4)]{1, 1, 0, 0, 1, 0;}
        \item[5)]{1, 1, 1, 0, 0, 0.}
      \end{enumerate}
    \end{minipage}
  \end{center}
 
The second and the fourth cases do not satisfy inequalities from the lemma $\ref{example}$. The third and the fifth case do not satisfy the condition $\widetilde{E}$. There is one possible set: $-\omega, \omega^{2}, \omega^{2}$. It satisfies the condition $\widetilde{E}$ with $f_{1} = 0, f_{5} = 3$.

Let us check whether the condition $\widetilde{RH}$. is met.

For $\sigma$ the values of $r^{*}_{\chi, u}$ are equal to $r_{\chi, u}$. 

For  $\sigma^{2}$ the values of $f_{1}$ and $f_{4}$ are equal to $4$ and $1$ respectively. Thus

$$r^{*}_{\chi, 2} = 1 - \sum\limits_{\substack{v \in I(6) \\ v \equiv 2 (mod \, 3)}}r^{*}_{\chi, 2}(\sigma) = 1 - 3 = -2.$$

We get that the condition is not met; thus, the automorphism is not met.
\end{proof}

\begin{lemma}\label{5m5}
Let C be a compact Riemann surface of genus g = 3, and  $\sigma$  be an automorphism of order~8. Then $J/\langle\sigma\rangle$ is not uniruled if and only if the eigenvalues of  $\sigma$ are $1, \zeta, \zeta^{3}$,  where $\zeta = e^{\frac{2\pi \mathbf{i}}{8}}$.
\end{lemma}

\begin{proof}

Let $1, \zeta, \mathbf{i}, \zeta^{3}, -1, \zeta^{5}, -\mathbf{i}, \zeta^{7}$ be the eigenvalues of  $\sigma$, where $\zeta = e^{\frac{2\pi \mathbf{i}}{8}}$. Their quantity is equal to $k_{0}$, $k_{1}$, $k_{2}$, $k_{3}$, $k_{4}$, $k_{5}$, $k_{6}$, $k_{7}$ respectively.

Let us write down the inequalities following from the Lefschetz formula for automorphisms $\sigma$, $\sigma^{2}$, $\sigma^{3}$ and inequality opposite to the global Reid condition:

\begin{equation*}
 \begin{cases}
$$2(k_{0} - k_{4}) + \sqrt{2}(k_{1} + k_{7} - k_{3} - k_{5}) \leq 2$$; \\
2(k_{0} + k_{4}) - 2(k_{2} + k_{6}) \leq 2; \\
k_{0} + k_{2} + k_{4} + k_{6} - k_{1} - k_{3} - k_{5} - k_{7} \leq 1; \\
k_{1} + 2k_{2} + 3k_{3} + 4k_{4} + 5k_{5} +6k_{6} + 7k_{7} < 8.

 \end{cases}
\end{equation*}

Let us write down all the possible sets of $k_{i}$, $i = 0, \ldots, 7$, satisfying the system:
\begin{table}[h]
\begin{center}
\caption{$N = 8$}
\label{tab:table5}
\begin{tabular}{|c|c|c|c|c|c|c|c|c|} 
\hline
$k_{0}$ & $k_{1}$ & $k_{2}$ & $k_{3}$ & $k_{4}$ & $k_{5}$ & $k_{6}$ & $k_{7}$ & |$\mathrm{Fix}_{C}(\sigma)$| \\
\hline
0 & 0 & 2 & 1 & 0 & 0 & 0 & 0 & $-$\\
\hline
0 & 1 & 0 & 2 & 0 & 0 & 0 & 0 & $-$\\
\hline
0 & 1 & 1 & 0 & 1 & 0 & 0 & 0 & $-$\\
\hline
0 & 1 & 1 & 1 & 0 & 0 & 0 & 0 & 2\\
\hline
0 & 1 & 2 & 0 & 0 & 0 & 0 & 0 & $-$\\
\hline
0 & 2 & 0 & 0 & 0 & 1 & 0 & 0 & $-$\\
\hline
0 & 2 & 0 & 0 & 1 & 0 & 0 & 0 & $-$\\
\hline
0 & 2 & 0 & 1 & 0 & 0 & 0 & 0 & $-$\\
\hline
1 & 0 & 0 & 2 & 0 & 0 & 0 & 0 & $-$\\
\hline
1 & 0 & 1 & 0 & 0 & 1 & 0 & 0 & $-$\\
\hline
1 & 0 & 1 & 1 & 0 & 0 & 0 & 0 & $-$\\
\hline
1 & 1 & 0 & 0 & 0 & 1 & 0 & 0 & 0\\
\hline
1 & 1 & 0 & 1 & 0 & 0 & 0 & 0 & 0\\
\hline
\end{tabular}
\end{center}
\end{table}

Remains three possible sets of eigenvalues:

\begin{center}
    \begin{minipage}{0.25\textwidth}
      \begin{enumerate}
        \item[1)]{0, 1, 1, 1, 0, 0, 0, 0;}
        \item[2)]{1, 1, 0, 0, 0, 1, 0, 0;}
        \item[3)]{1, 1, 0, 1, 0, 0, 0, 0.}
      \end{enumerate}
    \end{minipage}
  \end{center}

The third case and squared second case do not satisfy the condition $\widetilde{E}$; thus, they are not realized. The example of the first one was provided in \cite{AB}. Namely the curve   $y^{2} = x^{8} - 1$ and automorphism $\sigma(x, y) = (\zeta x, y)$. In this case, the basis of $H^{0, 1}(C, K_{C})$ will be $\frac{dx}{y}, \frac{xdx}{y}, \frac{x^{2}dx}{y}$. The eigenvalues are equal to  $1, 2, 3$.
\end{proof}

\begin{lemma}\label{5m6}
Let C be a compact Riemann surface of genus g = 3, and $\sigma$ is an automorphism of order~12. Then $J/\langle\sigma\rangle$ is uniruled if and only if eigenvalues are either $\zeta, \zeta^{3}, \zeta^{5}$ or $\zeta, \zeta^{2}, \zeta^{5}$.
\end{lemma}

\begin{proof}
Let $1, \zeta,  \zeta^{2}, \zeta^{3}, \zeta^{4}, \zeta^{5}, \zeta^{6}, \zeta^{7}, \zeta^{8}, \zeta^{9}, \zeta^{10}, \zeta^{11}$  be the eigenvalues of  $\sigma$, where $\zeta = e^{\frac{2\pi \mathbf{i}}{12}}$. Their quantity is equal to $k_{0}$, $k_{1}$, $k_{2}$, $k_{3}$, $k_{4}$, $k_{5}$, $k_{6}$, $k_{7}$, $k_{8}$, $k_{9}$, $k_{10}$, $k_{11}$ respectively.

Let us write down the inequalities following from the Lefschetz formula for automorphisms $\sigma$, $\sigma^{2}$, $\sigma^{3}$ and inequality opposite to the global Reid condition: 

\begin{equation*}
 \begin{cases}
$$2k_{0} + 2(k_{1} + k_{11})\mathrm{cos}(\frac{2\pi}{12}) + 2(k_{2} + k_{10})\mathrm{cos}(\frac{4\pi}{12})+ 2(k_{4} + k_{8})\mathrm{cos}(\frac{8\pi}{12})+ 2(k_{5} + k_{7})\mathrm{cos}(\frac{10\pi}{12}) - 2k_{6} \leq 2$$; \\
2(k_{0} + k_{6}) + 2(k_{1} + k_{7} + k_{5} + k_{11})\mathrm{cos}(\frac{4\pi}{12}) + 2(k_{2} + k_{8} + k_{4} + k_{10})\mathrm{cos}(\frac{8\pi}{12}) - 2k_{3} + 2k_{6} \leq 2; \\
2(k_{0} + k_{4} + k_{8}) - 2(k_{2} + k_{6} + k_{10}) \leq 2;\\
2(k_{0} + k_{3} + k_{6} + k_{9}) + 2(k_{1} + k_{2} + k_{4} + k_{5} + k_{7} + k_{8} + k_{10} + k_{11})\mathrm{cos}(\frac{8\pi}{12}) \leq 2\\
k_{0} + k_{2} + k_{4} + k_{6}+ k_{8} + k_{10}   - k_{1} - k_{3} - k_{5} - k_{7} - k_{9}  - k_{11}\leq 1; \\
k_{1} + 2k_{2} + 3k_{3} + 4k_{4} + 5k_{5} +6k_{6} + 7k_{7}+ 8k_{8}+ 9k_{9}+ 10k_{10}+ 11k_{11} < 12.

 \end{cases}
\end{equation*}

Let us write down all the possible sets of $k_{i}$, $i = 0, \ldots, 11$, satisfying the system:

\begin{table}[h]
\begin{center}
\caption{$N = 12$}
\label{tab:table6}
\begin{tabular}{|c|c|c|c|c|c|c|c|c|c|c|c|c|c|} 
\hline
$k_{0}$ & $k_{1}$ & $k_{2}$ & $k_{3}$ & $k_{4}$ & $k_{5}$ & $k_{6}$ & $k_{7}$ &$k_{8}$ &$k_{9}$ &$k_{10}$ & $k_{11}$ &|$\mathrm{Fix}_{C}(\sigma)$| \\
\hline
0 & 0 & 1 & 0 & 1 & 1 & 0 & 0 & 0 & 0 & 0 & 0 & $-$\\
\hline
0 & 0 & 1 & 1 & 0 & 1 & 0 & 0 & 0 & 0 & 0 & 0 & $-$\\
\hline
0 & 0 & 1 & 1 & 1 & 0 & 0 & 0 & 0 & 0 & 0 & 0 & 2\\
\hline
0 & 0 & 2 & 0 & 0 & 0 & 0 & 1 & 0 & 0 & 0 & 0 & $-$\\
\hline
0 & 0 & 2 & 0 & 0 & 1 & 0 & 0 & 0 & 0 & 0 & 0 & $-$\\
\hline
0 & 1 & 0 & 1 & 0 & 0 & 0 & 1 & 0 & 0 & 0 & 0 & 2\\
\hline
0 & 1 & 0 & 1 & 0 & 1 & 0 & 0 & 0 & 0 & 0 & 0 & 2\\
\hline
0 & 1 & 1 & 0 & 0 & 0 & 0 & 0 & 1 & 0 & 0 & 0 & $-$\\
\hline
0 & 1 & 1 & 0 & 0 & 0 & 0 & 1 & 0 & 0 & 0 & 0 & 1\\
\hline
0 & 1 & 1 & 0 & 0 & 1 & 0 & 0 & 0 & 0 & 0 & 0 & 1\\
\hline
0 & 1 & 1 & 0 & 1 & 0 & 0 & 0 & 0 & 0 & 0 & 0 & $-$\\
\hline
\end{tabular}
\end{center}
\end{table}

There remain five possible sets of eigenvalues:

\begin{center}
    \begin{minipage}{0.35\textwidth}
      \begin{enumerate}
        \item[1)]{0, 0, 1, 1, 1, 0, 0, 0, 0, 0, 0, 0;}
        \item[2)]{0, 1, 0, 1, 0, 0, 0, 1, 0, 0, 0, 0;}
        \item[3)]{0, 1, 0, 1, 0, 1, 0, 0, 0, 0, 0, 0;}
        \item[4)]{0, 1, 1, 0, 0, 0, 0, 1, 0, 0, 0, 0;}
        \item[5)]{0, 1, 1, 0, 0, 1, 0, 0, 0, 0, 0, 0.}
      \end{enumerate}
    \end{minipage}
  \end{center}

The example of the fifth case was provided in \cite{AB}, namely the curve is  $y^{3} = x(x^{3} - 1)$ and automorphism $\sigma(x, y) = (\zeta^{3}x, \zeta y)$. In this case the basis of $H^{0}(C, K_{C})$ is $\frac{dx}{y^{2}}, \frac{xdx}{y^{2}}, \frac{dx}{y}$, and eigenvalues are equal to $1, 2, 5$. The first, second and forth cases do not satisfy the condition $\widetilde{E}$. It remains to check the condition $\widetilde{RH}$ for the third case. Let us find values of $r^{*}_{\chi, u}$ for each degree of automorphism $\sigma$.

\begin{center}
    \begin{minipage}{0.35\textwidth}
      \begin{enumerate}
        \item[$\sigma^{2}$]{$: r^{*}_{\chi, 1} = 0, r^{*}_{\chi, 5} = 0$;} 
        \item[$\sigma^{3}$]{$: r^{*}_{\chi, 1} = 0, r^{*}_{\chi, 3} = 0$;} 
        \item[$\sigma^{4}$]{$: r^{*}_{\chi, 1} = 0, r^{*}_{\chi, 2} = 4$;} 
        \item[$\sigma^{6}$]{$: r^{*}_{\chi, 1} = 3.$} 
      \end{enumerate}
    \end{minipage}
  \end{center}
The condition  $\widetilde{RH}$ is met; thus, the automorphism is realized.
\end{proof}

Let us prove Theorem $\ref{m3}$.

\begin{proof}

Consider all possible orders of $N$. According to Theorem \ref{order} value $N$ can vary from $2$ to $14$. We have already studied the following cases and it has been shown that there are only automorphisms of order  2, 7, 8, 12 such that the quotient is uniruled: 

$N = 2, 3, 5, 7, 11, 13$ \, lemma $\ref{5m1};$

$N = 4:$ \,  lemma $\ref{5m3};$

$N = 6:$ \,  lemma $\ref{5m4};$

$N = 8:$ \,  lemma $\ref{5m5};$

$N = 12:$ lemma $\ref{5m6}.$

 Let us consider the remaining cases.\\

Let $N = 10$. According to Theorem \ref{max_ord} (paragraph 3) there are no automorphisms of order 10 on a compact Riemann surface of genus 3. \\

Let $N = 9$. According to Theorem \ref{max_ord}  (paragraph 3) there is only one compact Riemann surface of genus 3:
	$$y^{3} = x(x^{3} - 1)$$
and an automorphism of order $9$:
	$$\sigma(x, y) = (\zeta^{3} x, \zeta y).$$
The basis of  $H^{0}(C, K_{C})$ is  $\frac{dx}{y^{2}}, \frac{xdx}{y^{2}}, \frac{dx}{y}$. Then the eigenvalues of $\sigma$ are $\zeta, \zeta^{2}, \zeta^{4}.$ They do not satisfy the global Reid condition. Thus $J/G$ is uniruled.

Let $N = 14$. According to Theorem \ref{max_ord}  (paragraph 1) there is only one compact Riemann surface of genus 3:
	$$y^{2} = x(x^{7} - 1)$$
and an automorphism of order 14:
	$$\sigma(x, y) = (\zeta^{2} x, \zeta y).$$
The basis of $H^{0}(C, K_{C})$ is $\frac{dx}{y}, \frac{xdx}{y}, \frac{x^{2}dx}{y}$. Then the eigenvalues of  $\sigma$ are $\zeta, \zeta^{3}, \zeta^{5}$. They do not satisfy the global Reid condition. Thus $J/G$ is uniruled.
\end{proof}

\section{Corollaries}\label{Cor}
 
In this section we will prove Corollaries \ref{B_K} and \ref{c1}.

Let us prove Corollary \ref{B_K}.

\begin{proof}
Let us prove the first assertion.
The automorphism group of Bring's curve is isomorphic to $S_{5}$. According to Theorem \ref{m2} the quotient is uniruled only if the group contains elements of orders $15$ and $18$.
Maximal order of an element in $S_{5}$ is $6$. Thus the quotient of Jacobian of Bring's curve is not uniruled.

Let us prove the second assertion. In automorphism group of Klein's curve all non-trivial elements have orders $2, 3, 4, 7$, see \cite[p. 3]{Atlas}.

The elements of the same order are conjugated since the group is simple. Thus it is sufficient to check the Reid condition for one element of each conjugacy class. According to Theorem \ref{m2} the quotient by elements of orders $3$ and $4$ is not uniruled.

Consider the case when the order of element $\sigma$ on $\mathbb{P}$ is equal to $2$. There is a fixed line under the action of an element of order  $2$ on $\mathbb{P}^{2}$. It intersects our curve. Therefore there is at least one fixed point under the action of an element of order $2$. According to Theorem \ref{m2} if quotient by element of order $2$ is uniruled, then eigenvalues are equal to $1, 1, -1$. Nonetheless, according to the Lefschetz formula, this action does not have any fixed points. Thus the quotient is not uniruled.

It remains to consider the case when the order of element is $7$. Let automorphism $\sigma$ be $\sigma(u, v, w) = (\zeta u, \zeta^{4} v, \zeta^{2} w)$. Let us check  whether the Reid condition is met. In order to do it, let us write down the basis of $H^{0}(C, K_{C})$. Let us use the coordinates $x = \frac{u}{w}$ and $y = \frac{v}{w}$ on $\mathbb{A}^{2}$. Then eigenbasis is  $$\frac{xdx}{x^{3} + 3y^{2}}, \,\,\,  \frac{y(ydx - xdy)}{y^{3} + 3x}, \,\,\,    \frac{dy}{1 + 3x^{2}y}.$$ 
The eigenvalues are $\zeta, \zeta^{2}, \zeta^{4}$, and so according to Theorem \ref{m3}  the quotient is not uniruled.

\end{proof}

Let us prove Corollary \ref{c1}.

\begin{proof}
In the case when $g = 4$ the finiteness of number of curves such that quotients are not uniruled follows directly from Theorem \ref{m1}.

Consider the set of curves $C_{a}$ defined by equations  $$y^{2} = (x^{2} - 1)(x^{2} - 4)(x^{2} - 9)(x^{2} - a^{2}),$$  and automorphism  $\sigma(x, y) = (\zeta x, \zeta y)$. The eigenbasis of $H^{0}(C_{a}, K_{C_{a}})$ is $\frac{dx}{y}, \frac{xdx}{y}, \frac{x^{2}dx}{y}$.  Thus the eigenvalues of $\sigma$ are equal to $1, 1, -1.$ According to Theorem \ref{m3} the quotient of the Jacobian is not uniruled. There is an infinite number of such curves that are not isomorphic to each other. Indeed, the branch points of $C_{a}$, namely $\pm 1, \pm2, \pm3, \pm a$, are uniquely defined by $C_{a}$. If $C_{a}$ is isomorphic to $C_{a'}$ then the set of branch points of $C_{a}$ is mapped to the set of branch points of $C_{a'}$ under some change of coordinates on $\mathbb{P}^{1}$. This  is possible only for a finite number of values of $a'$. 

In the case when $g = 2$, consider the curves $C_{a}$ defined by  $$y^{2} = x(x^{2} - 1)(x^{2} - a^{2})$$ and automorphism $\sigma(x, y) = (-x, \mathbf{i}y)$. This automorphism has order $4$. According to Theorem \ref{b3} the quotient of the Jacobian is not uniruled. There is an infinite number of such curves that are not isomorphic to each other. Indeed, the  branch points of $C_{a}$, namely $0, \pm 1, \pm a, \infty$, are uniquely defined by $C_{a}$. If $C_{a}$ is isomorphic to $C_{a'}$ then   the set of branch points of $C_{a}$ is mapped to the set of branch points of $C_{a'}$ under some change of coordinates on $\mathbb{P}^{1}$. This  is possible only for a finite number of values of $a'$.
\end{proof}

\end{document}